%% file: main.tex
\documentclass[12pt]{article}

\usepackage[english]{babel}
\usepackage{amsfonts}
\usepackage{amsmath}
\usepackage{amssymb}
\usepackage{amsthm}
\usepackage{xcolor}
\usepackage{esint} 
\usepackage[utf8]{inputenc}
\usepackage{graphicx}
\usepackage{lipsum}
\newcommand\blfootnote[1]{%
  \begingroup
  \renewcommand\thefootnote{}\footnote{#1}%
  \addtocounter{footnote}{-1}%
  \endgroup
}

\usepackage[letterpaper,top=2cm,bottom=2cm,left=3cm,right=3cm,marginparwidth=1.75cm]{geometry}

\newtheorem{theorem}{Theorem}[section]
\newtheorem*{theorem*}{Theorem}
\newtheorem{proposition}[theorem]{Proposition}

\newtheorem{corollary}[theorem]{Corollary}
\theoremstyle{definition}

\newtheorem{remark}{Remark}

\def \MM {\mathbb{M}}

\def \SS {\mathbb{S}}

\def \Om {\Omega}

\title{Isoperimetric profile function comparisons with Integral Ricci curvature bounds }
\author{Jihye Lee, Fabio Ricci}
\date{}

\begin{document}
\maketitle

\input{FinalV}


\bibliographystyle{plain}
\end{document}

%% file: FinalV.tex
\begin{abstract}
   We prove comparison results for the Isoperimetric profile function in the setting of manifolds with $L^p$ integral bounds on the Ricci curvature. We extend previous work of Ni and Wang \cite{Ni} and Bayle and Rosales \cite{bayle2} under the usual pointwise bounds for the Ricci curvature.
\end{abstract}
\blfootnote{Jihye Lee: UC Santa Barbara, Department of Mathematics, email: jihye@ucsb.edu. }
\blfootnote{Fabio Ricci: UC Santa Barbara, Department of Mathematics, email: FabioRicci@ucsb.edu. }

\section{Introduction}
Isoperimetric inequalities stand as classical pillars in geometric analysis, providing fundamental insights into the intricate balance between volume and surface area constraints within geometric spaces. Central to the study of these inequalities is the concept of the isoperimetric profile function, with explicit lower bounds representing foundational isoperimetric inequalities. These have been well studied; notably for positive sectional curvature, we have the well-known Levy-Gromov inequality \cite{Gromov} and B\'{e}rard-Besson-Gallot generalization \cite{gallot}. Later, Morgan and Johnson noted that the isoperimetric profile function must satisfy a differential inequality \cite{Morgan}, the argument was then refined by Bayle and Rosales \cite{bayle1}, \cite{bayle2}. The consequences are immediate: they were able to give a new proof of Gromov's classical result and produce new sharp comparison theorems, by integrating such a differential inequality. Later, Ni and Wang \cite{Ni} found an alternative proof of the Morgan and Johnson's result as well as new differential inequalities for the isoperimetric profile function. This has been extended in the nonsmooth setting of $RCD(K,N)$ spaces by Antonelli, Pasqualetto, Pozzetta, and Semola \cite{antonelli}. Our work resides within the smooth setting and requires an integral curvature bound as opposed to the usual pointwise lower bound. We closely follow and extend some of the results of Ni and Wang \cite{Ni}, controlling the error term in terms of the integral curvature excess. We also discuss the case of the relative isoperimetric problem in the same spirit as Bayle and Rosales \cite{bayle2}.

Let $(M,g)$ be an $n$-dimensional Riemannian manifold. The isoperimetric profile function $h_2(\beta,g)$ is defined on the interval $(0, |M|)$ as follows:
\[ h_2(\beta,g) := \inf\{|\partial \Omega|  \, | \, \Omega \text{ is a smooth domain in } M \text{ with }|\Omega| = \beta\} \]
Here, $|\Omega|$ denotes the $n$-dimensional volume, and $|\partial \Omega|$ denotes the $(n-1)$-dimensional area of $\partial \Omega$. In our definition, we allow $|M| = \infty$. A region $\Omega$ in $M$ with $|\partial \Omega| = h_2(\beta,g)$ and $|\Omega| = \beta$ is termed an isoperimetric region of volume $\beta$. The existence and regularity of such regions are well understood; for detailed discussions, refer to Chapter 6 of \cite{Sakai}. In the context of a compact Riemannian manifold $M$, we can define the normalized isoperimetric profile function $h_1(\beta, g)$. This function is defined for any $\beta \in (0,1)$, 
\[ h_1(\beta,g) := \inf \left\{ \frac{|\partial \Omega|}{|M|} \left|  \, \Omega \text{ is a smooth domain in }M  \text{ with } \frac{|\Omega|}{|M|} = \beta\right.\right\}. \]
It is well known that $h_1(\beta,g)$ is H\"older continuous and, therefore, continuous. Moreover, it satisfies the following asymptotic:
\begin{align}\label{eq:asymptotics}
    \lim_{\beta \rightarrow 0}\frac{h_1(\beta,g)}{\beta^{\frac{n-1}{n}}} = n\left(\frac{\omega_n}{|M|}\right)^{\frac{1}{n}}.
 \end{align}
Here, $\omega_n$ denotes the $n$-dimensional volume of the unit ball in Euclidean space $\mathbb{R}^n$. Then, for a compact manifold $M$, the following relationship holds:
\begin{align}\label{eq:h1h2relation}
    h_1(\beta, g) = \frac{h_2(\beta |M|,g)}{|M|}.
\end{align}
In this paper, we establish comparisons for both the $h_1$ and $h_2$ isoperimetric profile functions. Beginning with $h_2$, we first revisit the result we are generalizing. Originally proven by Morgan and Johnson \cite{Morgan}, Ni and Wang \cite{Ni} later presented an alternative argument, which we extend to incorporate integral curvature bounds.
\begin{theorem*}[Morgan and Johnson \cite{Morgan}, Ni and Wang \cite{Ni}]
    Let $(M^n,g)$ be a complete Riemannian manifold. For a given $k\in\mathbb{R}$ assume that $\mathrm{Ric}\geq (n-1)k$.
    Then for $\beta \in (0,|M|)$
    \begin{equation}\label{eq:MJ}
       h_2(\beta,g) \leq h_2(\beta,g_k),
    \end{equation}
    where $(\MM^n_k,g_k)$ is a complete simply connected space of constant curvature $k$.
    If the equality ever holds somewhere, $(M^n,g)$ must be isometric to $(\MM^n_k,g_k)$.
\end{theorem*}

Notice how the Levy-Gromov isoperimetric inequality \cite{Gromov} can be reformulated as $h_1(\beta,g)\geq h_1(\beta, g_1)$, where $g_1$ is the standard metric on the unit sphere. By combining this with the relation \eqref{eq:h1h2relation} between $h_1$ and $h_2$, as well as the inequality \eqref{eq:MJ}, we obtain the following (Corollary 3.2 in \cite{Ni}):
 \begin{align}\label{eq:Ineq}
     h_1(\beta,g_1)\leq h_1(\beta, g)\leq \frac{|{\mathrm{\SS}}^n|}{| M|}  \cdot h_1 \left( \frac{|M|}{|{\mathrm{\SS}}^n|}\beta,g_1\right).
 \end{align}
The integral Ricci curvature measures the amount of Ricci curvature that falls below $(n-1)k$ in the $L^p$ sense, where $k$ is a given constant in $\mathbb{R}$. Let $\rho(x)$ denote the smallest eigenvalue of the Ricci tensor at $x \in M$. We define $\mathrm{Ric}_-^k(x) = \max\{0,(n-1)k - \rho(x)\}$, and the integral Ricci curvature is defined as follows:
\begin{align*}
\| \mathrm{Ric}_-^k \|_p = \left( \int_M (\mathrm{Ric}_-^k(x))^p \, d\mathrm{vol} \right)^\frac{1}{p}.
\end{align*}
Integral curvature has been extensively studied in the last few decades. In particular, in this work, we will rely on the following classic results: the extension of the Laplace comparison to this setting by Petersen and Wei \cite{PetersenWei}, the improved volume comparison of Chen and Wei \cite{chenwei}, the generalization of the Heintze-Karcher inequality by Petersen and Sprouse \cite{Petersen-Sprouse}, and the extension of the Bonnet-Myers theorem by Aubry \cite{Aubry07}. Isoperimetric constant estimates have also been explored in the context of integral curvature bounds; see, for example, Gallot \cite{gallot2} and Dai, Wei, and Zhang \cite{wei2}.

 Our first result pertains to the non-compact setting where $k<0$. This is particularly noteworthy due to the scarcity of results regarding integral curvature for non-compact manifolds compared to the compact case. In the compact setting for $k>0$, our result can be used to assess the degree to which we need to relax the second inequality in \eqref{eq:Ineq} when transitioning from pointwise to integral curvature conditions.
\begin{theorem}\label{them:1.1negative}
    Let $(M^n,g)$ be a Riemmanian manifold, $p>\frac{n}{2}$ and  $k \leq 0$ be given. When $k=0$ assume that $\mathrm{diam}(M) = d < \infty$. Then for any $\beta \in (0, |M|)$, we have
\begin{align*}
    h_2(\beta,g) - h_2(\beta, g_k) &\leq \left( \frac{(n-1)(2p-1)}{2p-n} \|\mathrm{Ric}_-^k\|_p\right)^\frac{1}{2} \beta^{\frac{2p-1}{2p}}+ f(\beta,n,p,k,d,\|\mathrm{Ric}_-^k\|_p),
\end{align*}
    where $(\MM_k^n,g_k)$ is a complete simply connected space of constant curvature $k$ and $f\equiv 0$ when $\|\mathrm{Ric}_-^k \|_p=0$ recovering \eqref{eq:MJ}.
\end{theorem}

\begin{theorem}\label{them:1.1}
Let $(M^n,g)$ be a Riemmanian manifold, $p>\frac{n}{2}$ and  $k>0$ be given. Assume additionally that $\mathrm{diam}(M) < \frac{\pi}{2 \sqrt{k}}$. Then there exist an $\epsilon >0 $ such that if  
$\|\mathrm{Ric}_-^k \|_p < \epsilon,$
then for any $\beta \in (0, |M|)$, we have
\begin{align*}
    h_2(\beta,g) - h_2(\beta, g_k) &\leq \left( \frac{(n-1)(2p-1)}{2p-n} \|\mathrm{Ric}_-^k\|_p\right)^\frac{1}{2} \beta^{\frac{2p-1}{2p}}+ f(\beta,n,p,k,d,\|\mathrm{Ric}_-^k\|_p),
\end{align*}
    where $(\MM_k^n,g_k)$ is a complete simply connected space of constant curvature $k$ and $f\equiv 0$ when $\|\mathrm{Ric}_-^k \|_p=0$ recovering \eqref{eq:MJ}.
\end{theorem}
\begin{remark}
We also notice that $f$ is an increasing function in $\beta$ and vanishes at $\beta =0$.
In the case of $k = 0$, we can write down $f$ explicitly: $$f=\left[(1+C(n,p,k,d) (\|\mathrm{Ric}_-^k\|_p )^\frac{1}{2})^{\frac{2p}{n}} - 1\right]n\left(\omega_n
    \right)^\frac{1}{n} \beta^{\frac{n-1}{n}}.$$ 
\end{remark}

We will prove these theorems by adapting the argument of Ni and Wang \cite{Ni} to incorporate the integral curvature condition. Specifically, we will investigate the discrepancy in measures of the boundaries of balls with fixed volumes, comparing those in our manifold with those in the model space. This comparison yields a natural upper bound for the $h_2$ functions. To calculate this difference, we will directly compute derivatives, resulting in the emergence of a mean curvature term. To control this term, we will apply the mean curvature comparison technique as outlined in Petersen and Wei \cite{PetersenWei}.

Define the usual comparison functions as follows,
$$sn_k(t) = \begin{cases}
    \frac{1}{\sqrt{k}}\sin (\sqrt{k}t), &k >0\\
    t & k =0\\
    \frac{1}{\sqrt{k}}\sinh (\sqrt{k}t) & k <0.
\end{cases}$$
Our next theorem involves the $h_1 $ isoperimetric profile function:
\begin{theorem}\label{thm:supsolution}
    Let $(M^n,g)$ be an $n$-dimensional Riemannian manifold, $p>\frac{n}{2}$ for any $k>0$ and $\alpha>1$, there is  $\delta=\delta(n,p,\alpha,k) >0$, such that if $\|\mathrm{Ric}_-^k\|_p\leq \delta$, then the isoperimetric profile function $h_1(\beta, g)$ is a positive super solution of
    \begin{equation}\label{eq:alpha-super}
        \alpha \psi \left(k + \left(\frac{\psi'}{n-1}\right)^2\right)^\frac{n-1}{2} = \frac{1}{\int_0^{d'} sn_k^{n-1}t \, dt},
    \end{equation}
    where $d' = \min\{\pi, \mathrm{diam}(M)\}$.
\end{theorem}
This generalization extends Theorem 4.1 of Ni and Wang \cite{Ni} and retrieves it by taking the limit as $\alpha \to 1^+$. It's worth noting that we didn't assume compactness, as it follows from the Bonnet-Myers theorem generalization by Aubrey \cite{Aubry07}. By combining Petersen and Sprouse's estimate \cite{Petersen-Sprouse} with Ni and Wang's approach, we also establish the following result, which replicates the estimate proposed by Berard, Besson, and Gallot in their work \cite{Ni}. 
\begin{theorem}\label{thm:1.3}
     Let $(M^n,g)$ be an $n$-dimensional Riemannian manifold, $p>\frac{n}{2}$, for any $k >0$  and $\alpha >1$, there is  $\delta=\delta(n , p,\alpha, k)>0$, such that if $\|\mathrm{Ric}_-^k\|_p  \leq \delta$, then
    $$h_1(\beta, g) \geq  L h_1(\beta,g_k)-\epsilon(n,d',\alpha,k),$$
    where $d' = \min \{\frac{\pi}{\sqrt{k}}, \mathrm{diam}(M)\}$, $L= \left(\frac{\gamma_n}{\lambda_{n,d'}^k}\right)^\frac{1}{n}$,  $\gamma_n = \int_0^{\frac{\pi}{\sqrt{k}}} sn_k^{n-1}t \, dt$, $\lambda_{n,d'}^k = \int_0^{d'} sn_k^{n-1}t\, dt$,  $\epsilon(n,d',\alpha,k) = \frac{(\alpha -1)L}{\alpha}\frac{1}{\int_0^{\frac{\pi}{\sqrt{k}}} \sin (\sqrt{k} t) \,dt}$, and $g_k$ is the standard metric on $\SS_k^n$.
\end{theorem}
Notice that since $L\geq 1 $, this improves the classical Levy-Gromov isoperimetric inequality, the first inequality in \eqref{eq:Ineq}.
Let us consider a complete Riemannian manifold $(M,g)$ with $\mathrm{Ric} \geq (n-1)k$ for $k>0$. Since $\|\mathrm{Ric}_-^k\|_p = 0$, we know that
     $$h_1(\beta, g) \geq  Lh_1(\beta,g_k)-\epsilon(n,d',\alpha,k)$$
     for all $\alpha >1$.
      Since $\lim\limits_{\alpha \to 1^+}\epsilon(n,d',\alpha,k) = 0$, we get the following corollary.
\begin{corollary}
Let $(M^n,g)$ be a complete Riemannian manifold with $\mathrm{Ric} \geq (n-1)k$ for $k >0$. Then for any $\beta \in (0,1)$
     $$h_1(\beta,g) \geq L h_1(\beta,g_k),$$
     which is Bérard–Besson–Gallot's estimate.
\end{corollary}
The paper is structured as follows: In Section 2, we concentrate on the $h_2$ isoperimetric profile function. Initially, we establish Theorem \ref{them:1.1negative} along with Theorem \ref{them:1.1}, followed by an exploration of its relative isoperimetric counterpart. In Section 3, our focus shifts to the $h_1$ profile function, where we prove Theorem \ref{thm:supsolution} and Theorem \ref{thm:1.3}.

\paragraph{Acknowledgements:} We are deeply grateful to our advisor Guofang Wei for her guidance and for introducing us to the Ni and Wang's paper \cite{Ni}. Her support and fruitful discussions were instrumental in the success of this project.

\section{Comparisons for $h_2$ isoperimetric profile} 
In this section we prove  the two results involving the $h_2$ isoperimetric profile function. We then move to the relative version of our $h_2$ comparison in the same spirit as \cite{bayle2}. 

\subsection{Proof of Theorem \ref{them:1.1negative} and Theorem \ref{them:1.1}}
We begin with a simple observation on positive increasing functions. This will be used, in the positive $k>0$ case when comparing the radius of a ball of the same size on a given manifold and model space. For the $k=0$ we will avoid this and instead use a direct computation. For the negative $k<0$ scenario we will need Proposition \ref{prop2}. 

\begin{proposition}\label{prop:1}
Let $A \in C^\infty\left([0,2r]\right)$ and $A$ be positive on $(0,2r)$. Assume that it is strictly increasing  on the subset $(0,r]\subseteq [0,2r]$. Let $ 1 \leq \alpha \leq 2$ and
assume that for all $s \in [0,r]$, we have $A(r+s) \geq A(r-s)$. If
$$\int_0^{\bar{r}} A(s) ds \leq \alpha \int_0^r A(s) ds,$$
then $\bar{r} \leq \alpha r$.
\end{proposition}
\begin{proof}
When $ \alpha =1$, since $A$ is positive, $\int_0^{\bar{r}} A(s) ds \leq \int_0^r A(s) ds$ implies $\bar{r} \leq r$. Now, consider the case $1 <\alpha \leq 2$.
Assume $\bar{r} > \alpha r$ to get a contradiction.
Then
$$\int_0^{\bar{r}} A(s) ds > \int_0^{\alpha r} A(s) ds = \int_0^r A(s) ds + \int_r^{\alpha r} A(s) ds.$$
Let $B(s) = A (r-s)$ on $[0,r]$. Then $B'(s) = - A' (r-s) < 0$. That is, $B$ is strictly decreasing. Since $1 < \alpha \leq 2$, we have $( \alpha -1 ) s < s$ and so $B(s) < B (( \alpha -1) s)$. Thus,
$$(\alpha -1 ) B (s) < (\alpha -1) B((\alpha -1) s )$$
for $0 < s < r$.
By integrating it, we have
\begin{align*}
    (\alpha -1 ) \int_0^r B(s) ds  
    & < \int_0^r (\alpha -1 ) B((\alpha-1)s) ds  = \int_0^{(\alpha -1 ) r} B(t) dt \\
    & = \int_0^{(\alpha -1 ) r} A (r-t) dt \leq \int_0^{(\alpha -1 ) r} A(r+t) dt = \int_r^{\alpha r} A(s) ds.
\end{align*}
Thus,
\begin{align*}
    \int_0^{\bar{r}} A(s) ds & > \int_0^r A(s) ds + (\alpha -1) \int_0^r B(s) ds\\
    & = \int_0^r A (s) ds + (\alpha -1 ) \int_0^r A(r-s) ds = \alpha \int_0^r A (s) ds,           
\end{align*}
which is a contradiction to our assumption. Thus, $ \bar{r} \leq \alpha r$.
\end{proof}
The next proposition will mirror the previous one but for the scenario where $k<0$. We will explicitly compute this in the space form of constant sectional curvature $k=-1$, with the general case following straightforwardly. Notably, since $\sinh(s)$ is an increasing function everywhere, we don't need the assumption $\alpha \leq 2$ as in Proposition \ref{prop:1}. 
\begin{proposition}\label{prop2}
    Denote $A(r)=\sinh^{n-1}(r)$ with $r\in [0,+\infty)$, let $\alpha \geq 1$ and assume that 
    $$\int_0^{\bar{r}} A(s) ds \leq \alpha \int_0^r A(s) ds.$$
    Then we have $\bar{r} \leq \alpha^{\frac{1}{n}} r$.
\end{proposition}
\begin{proof}
    Defining the function $f(r)=\frac{\int_0^{r} A(s) ds}{r^n}$, our assumption can be rewritten as 
    $$f(\bar r) \leq \alpha f(r) \frac{r^n}{\bar r ^n}.$$ 
    By directly computing the derivative, we see that  $f$ is strictly increasing. Consider the case where $0<r<\bar r$, we conclude that $f(r) < f(\bar r) \leq \alpha f(r) \frac{r^n}{\bar r ^n}$ implying $1<\alpha \frac{r^n}{\bar r ^n}$, which leads to the conclusion $\bar r ^n < \alpha r^n$. The case $\bar r \leq r$ follows easily, since $\alpha \geq 1$, resulting in $\bar r ^n \leq \alpha r^n$ and thus the desired conclusion of $\bar r \leq \alpha^\frac{1}{n}r$
\end{proof}
\begin{remark}
    In the case $k = 0$, let $A(s)  = s^{n-1}$. By direct computation, $\int_0^{\bar r} A(s) \, ds \leq \alpha \int_0^r A(s) ds$ implies that  $\bar r \leq \alpha^\frac{1}{n} r$.
\end{remark}
We are now ready to prove Theorem \ref{them:1.1}. The proof strategy entails considering the difference between the measures of the boundaries of balls in our manifold and those in the model space. We choose this as it offers a straightforward upper bound for the difference of the $h_2$ functions. To estimate this difference, we will explicitly take the derivatives, naturally leading to the emergence of a mean curvature term. We will adopt a strategy similar to that of Petersen and Wei \cite{PetersenWei} to control such a term, in conjunction with their volume comparison technique for integral curvature.

\begin{proof}[Proof of Theorem \ref{them:1.1negative} and Theorem \ref{them:1.1}] 
Fix a point $x \in M$ and $\beta \in ( 0 , |M|)$. Let us define a function $I_x (t,g) = |\partial B_x(r)|$, where $r$ is defined as $|B_x(r)| = t$. Then we have
$$h_2 (\beta, g) =\inf\limits_{|\Omega|=\beta} |\partial \Omega| \leq I_x(\beta, g).$$
Note that $h_2 (\beta, g_k) = I_{\bar{x}}(\beta, g_k)$ for $\bar{x} \in \MM_k$. By the above two inequalities, we have
$$h_2 (\beta, g ) - h_2(\beta, g_k) \leq I_x (\beta ,g ) - I_{\bar{x}}(\beta, g_k).$$

Let $D(t) = I_x (t, g) - I_{\bar{x}} (t, g_k)$. We now want to find an upper bound of $D(\beta)$.
Recall that $I_x(t,g) = |\partial B_x(r)|$ with $|B_x(r)|=t$. Then,
$$\frac{d}{dr} |\partial B_x(r)| = \frac{d}{dr} \int_{\SS^{n-1}} A(r,\theta) d \theta = \int_{\SS^{n-1}}m(r,\theta) A(r,\theta) d \theta $$
and
$$\frac{dt}{dr} = \frac{d}{dr}|B_x(r)| =\frac{d}{dr} \int_0^r\int_{\SS^{n-1}} A(s,\theta) d \theta ds = \int_{\SS^{n-1}}A ( r, \theta) d \theta = |\partial B_x(r)|,$$
where $m(r,\theta)$ is the mean curvature and $A(s,\theta) d\theta ds$ is a volume form on manifold $M$ when we use a geodesic polar coordinates centered at $x$.
Thus,
$$I_x'(t, g)= \frac{d}{dt}|\partial B_x(r)| = \frac{d}{dr}|\partial B_x(r)| \frac{dr}{d t} =  \frac{\int_{\SS^{n-1}} m(r,\theta) A(r,\theta ) d \theta}{|\partial B_x(r)|}$$
and
$$I_{\bar{x}}'(t,g_k) =\frac{\int_{\SS^{n-1}} \bar{m}(\bar r) \bar{A}(\bar r) d \theta}{|\partial B_{\bar{x}}(\bar r)|}=\bar{m}(\bar{r}),$$
where $B_{\bar{x}}(\bar{r})$ is the ball on $\MM_k$ such that $|B_{\bar{x}}(\bar{r})| = t$.
Then
\begin{align*}
    D'(t) &= I_x'(t,g) - I_{\bar{x}}'(t,g_k)\\
    &=\frac{\int_{\SS^{n-1}} m(r,\theta) A(r,\theta ) d \theta}{|\partial B_x(r)|} -  \bar{m}(\bar{r})\\
    &= \frac{\int_{\SS^{n-1}} m(r,\theta) A(r,\theta ) d \theta}{|\partial B_x(r)|} - \frac{\int_{\SS^{n-1}} \bar{m}(\bar{r}) A(r,\theta ) d \theta}{|\partial B_x(r)|}\\
    & = \frac{\int_{\SS^{n-1}} [m(r,\theta)-\bar{m}(r)] A(r,\theta ) d \theta}{|\partial B_x(r)|} + \frac{\int_{\SS^{n-1}} [\bar{m}(r)-\bar{m}(\bar{r})] A(r,\theta ) d \theta}{|\partial B_x(r)|}\\
    & = \frac{\int_{\SS^{n-1}} [m(r,\theta)-\bar{m}(r)] A(r,\theta ) d \theta}{|\partial B_x(r)|} +  [\bar{m}(r)-\bar{m}(\bar{r})].
\end{align*}
Let $m_+^k(r,\theta):= (m(r,\theta) -\bar{m}(r))_+$.
Then
\begin{align*}
   \frac{\int_{\SS^{n-1}} [m(r,\theta)-\bar{m}(r)] A(r,\theta ) d \theta}{|\partial B_x(r)|} \leq\frac{\int_{\SS^{n-1}} m_+^k(r,\theta) A(r,\theta ) d \theta}{|\partial B_x(r)|}.
\end{align*}
Since $I_x(0,g) = 0$ and $I_{\bar{x}}(0,g_k)=0$, we know that $D(0) = 0$. Thus,
\begin{align*}
    D(\beta)  = \int_0^\beta D'(t) d t 
    &\leq \int_0^\beta \frac{\int_{\SS^{n-1}} m_+^k(r,\theta) A(r,\theta ) d \theta}{|\partial B_x(r)|} dt + \int_0^\beta (\bar{m}(r) - \bar{m}(\bar{r})) d t\\
    &\leq \int_0^{r_0} \int_{\SS^{n-1}} m_+^k(r,\theta) A(r,\theta ) d \theta dr + \int_0^{\beta} (\bar{m}(r) - \bar{m}(\bar{r})) d t,
\end{align*}
where $r_0$ is determined by $\beta = |B_x(r_0)|$.
Using H\"{o}lder's inequality, we have
\begin{align*}
    &\int_0^{r_0}\int_{\SS^{n-1}} m_+^k (r , \theta) A(r,\theta ) d \theta dr\\
    &\leq \left( \int_0^{r_0}\int_{\SS^{n-1}} (m_+^k(r, \theta) )^{2p} A(r,\theta ) d \theta dr\right)^\frac{1}{2p} \left( \int_0^{r_0}\int_{\SS^{n-1}}  A(r,\theta ) d \theta dr\right)^\frac{2p-1}{2p}\\
    & =\left( \int_{B_x(r_0)} (m_+^k)^{2p} d vol\right)^\frac{1}{2p}|B_x(r_0)|^\frac{2p-1}{2p}
\end{align*}
We now want to apply the mean curvature comparison under integral Ricci curvature condition \cite{PetersenWei}:
$$\|m_+^k\|_{2p}(r) =\sup_{y \in M}\left(\int_{B_y(r)}(m_+^k)^{2p} d vol\right)^\frac{1}{2p}\leq\left( \frac{(n-1)(2p-1)}{2p-n} \|\mathrm{Ric}_-^k\|_p(r)\right)^\frac{1}{2}. $$
Combining the above inequalities, we have
\begin{align*}
    D(\beta) &\leq  \|m_+^k\|_{2p}(r_0)|B_{x}(r_0)|^\frac{2p-1}{2p} + \int_0^{\beta} (\bar{m}(r) - \bar{m}(\bar{r})) d t\\
    &\leq \left( \frac{(n-1)(2p-1)}{2p-n} \|\mathrm{Ric}_-^k\|_p(r_0)\right)^\frac{1}{2} \beta^{\frac{2p-1}{2p}} + \int_0^{\beta} (\bar{m}(r) - \bar{m}(\bar{r})) dt.
\end{align*}
Finally,
\begin{align}\label{mid-result}
    h_2 (\beta, g ) - h_2( \beta, g_k) \nonumber&\leq I_x (\beta ,g ) - I_{\bar{x}}(\beta, g_k)
    \nonumber\\
    &\leq \left( \frac{(n-1)(2p-1)}{2p-n} \|\mathrm{Ric}_-^k\|_p\right)^\frac{1}{2} \beta^{\frac{2p-1}{2p}} + \int_0^{\beta} (\bar{m}(r) - \bar{m}(\bar{r})) d t.
\end{align}

 In order to estimate the second term in the last inequality, we want to compare $r$ and $\bar r$ by using the volume comparison. Recall the improved volume comparison theorem with integral curvature condition as in \cite{chenwei}:
$$|B_{\bar{x}}(\bar{r})|=|B_x(r)| \leq (1+C(n,p,k,r) (\|\mathrm{Ric}_-^k\|_p (r))^\frac{1}{2})^{2p} |B_{\bar{x}} (r)|,$$
where the constant $C(n,p,k,r)$ is
\begin{align*}
    C(n,p,k,r)=\left(  \frac{(n-1)(2p-1)}{2p-n}  \right)^\frac{1}{2}\int_0^r \left(\frac{1}{|B_{\bar{x}} (r)|} \right)^\frac{1}{2p}dt.
\end{align*}
When $k<0$, this constant is bounded for all $r$, eliminating the need for compactness in this scenario. We denote $C(n,p,k,d)$ as the limit of such a constant as $r$ approaches infinity. Conversely, for the case when $k=0$, compactness becomes necessary as $C$ increases with respect to $r$. In the positive case of $k>0$, we already assume a stronger restriction on the diameter to apply the mean curvature and volume comparison.

Notice that $\|\mathrm{Ric}_-^k\|_p (r)$ is also an increasing function of $r$, we can then conclude
$$|B_{\bar{x}}(\bar{r})| \leq (1+C(n,p,k,d) (\|\mathrm{Ric}_-^k\|_p )^\frac{1}{2})^{2p} |B_{\bar{x}} (r)|,$$
where $d = \mathrm{diam}(M)$.
Now we get
\begin{align}\label{eq:prop}
    \int_0^{\bar r} \bar A(s) ds \leq (1+C(n,p,k,d) (\|\mathrm{Ric}_-^k\|_p )^\frac{1}{2})^{2p} \int_0^{r} \bar A(s) ds.
\end{align}
Now depending on the sign of $k$ we either apply Proposition 2.1 or Proposition 2.2. Define $q=2p$ if $k>0$ and $q=\frac{2p}{n}$ if $k\leq 0$. Consider first the positive $k>0$. Then there is $\epsilon>0$ such that $\|\mathrm{Ric}_-^k\|_p < \epsilon$ implies $(1+C(n,p,k,d) (\|\mathrm{Ric}_-^k\|_p )^\frac{1}{2})^{2p} \leq 2$.
Since $\bar A (s)$ is positive increasing function on $(0,r]$, we can use Proposition \ref{prop:1}. Similarly for negative $k<0$ we use Proposition 2.2 and this time we don't require $\|\mathrm{Ric}_-^k\|_p$ to be small and obtain  
\begin{align}\label{ineq:r}
    \bar{r} \leq (1+C(n,p,k,d) (\|\mathrm{Ric}_-^k\|_p )^\frac{1}{2})^{q} r.
\end{align}
For $k=0$ the above inequality follows from a direct computation.

Since the mean curvature $\bar{m}$ on model space is decreasing in $r$, we have
$$\bar{m}(r) \leq \bar{m}\left(\frac{\bar{r}}{(1+C(n,p,k,d) (\|\mathrm{Ric}_-^k\|_p )^\frac{1}{2})^{q}}\right).$$
Thus,
$$\int_0^{\beta} (\bar{m}(r) - \bar{m}(\bar{r})) d t \leq \int_0^\beta\left(\bar{m}\left(\frac{\bar{r}}{(1+C(n,p,k,d) (\|\mathrm{Ric}_-^k\|_p )^\frac{1}{2})^{q}}\right) - \bar{m}(\bar{r} )\right) dt.$$
Let
$$f(\beta,n,p,k,d,\|\mathrm{Ric}_-^k\|_p): = \int_0^\beta\left(\bar{m}\left(\frac{\bar{r}}{(1+C(n,p,k,d) (\|\mathrm{Ric}_-^k\|_p )^\frac{1}{2})^{q}}\right) - \bar{m}(\bar{r} )\right) dt.$$
We can now conclude that
$$h_2(\beta,g) - h_2(\beta,\bar{g})  \leq \left( \frac{(n-1)(2p-1)}{2p-n} \|\mathrm{Ric}_-^k\|_p\right)^\frac{1}{2} \beta^{\frac{2p-1}{2p}} + f(\beta, n , p , k , d, \|\mathrm{Ric}_-^k\|_p).$$
\end{proof}

\subsection{Relative isoperimetric profile function comparison}
 We now move on to discuss the relative isoperimetric comparison for convex bodies. The difference here is that the perimeter is considered \textit{relative} to the convex body $\Omega$ which is a convex smooth domain with compact closure in $(M^n,g)$. If a hypersurface splits $\Omega$ into two sets, the relative perimeter of each of these sets is the surface area of the hypersurface. In other words, pieces from the boundary of $\Omega$ don't contribute to the relative perimeter. 
 
 The differential inequality for a relative isoperimetric profile function on a convex domain with pointwise Ricci curvature lower bounds was obtained by Bayle and Rosales \cite{bayle2}, following ideas from Bayle \cite{bayle1} for manifolds without boundary. Rather than following Bayle and Rosales' approach, we will modify the method we used in Theorem \ref{them:1.1negative} and Theorem \ref{them:1.1} in order to get a relative isoperimetric profile function comparison with integral Ricci curvature. For a more extensive discussion of the relative isoperimetric problem, see for example Chapter 9 of \cite{ritorebook}.
Define the relative isoperimetric profile function with respect to a convex body $\Omega$ as
$$h_2^\Om(\beta,g) = \inf \{|\partial D \cap \Om| \, | \, D \text{ is a smooth domain in } \Om \text{ with } |D \cap \Omega| = \beta\}.$$
Here we emphasize that $\partial D \cap \partial \Omega$ could be not empty.
Denote by $\mathbb{H}_k^n$ an half space in the simply connected space form $(\MM_k, g_k)$ of constant sectional curvature $k$, and denote with $h_2^{\mathbb{H}_k^n}(\beta,g_k)$ its relative isoperimetric function with respect to $\mathbb{H}_k^n$.
\begin{theorem}\label{thm:relative1}
Let $(M^n,g)$ be a Riemmanian manifold and $\Om$ be a convex body on $M$. Let $p>\frac{n}{2}$ and $k > 0$ be given. Assume that $\mathrm{diam}(M) =d < \frac{\pi}{2 \sqrt{k}}$. Then there is $\epsilon>0$ such that if $(M^n,g)$ satisfies
$\|\mathrm{Ric}_-^k \|_p < \epsilon,$
then for any $\beta \in (0, |\Omega|)$, we have
\begin{align*}
    h_2^\Om(\beta,g) - h_2^{\mathbb{H}_k^n}(\beta, g_k) &\leq \left( \frac{(n-1)(2p-1)}{2p-n} \|\mathrm{Ric}_-^k\|_p\right)^\frac{1}{2} \beta^{\frac{2p-1}{2p}}+ \tilde{f}(\beta,n,p,k,d,\|\mathrm{Ric}_-^k\|_p).
\end{align*}
\end{theorem}

\begin{theorem}\label{thm:relative2}
     Let $(M^n,g)$ be a Riemmanian manifold and $\Om$ be a convex body on $M$. Let $p>\frac{n}{2}$ and  $k \leq  0$ be given. When $k=0$ assume that $\mathrm{diam}(M) = d < \infty$. Then for any $\beta \in (0, |\Omega|)$, we have
\begin{align*}
    h_2^\Om(\beta,g) - h_2^{\mathbb{H}_k^n}(\beta, g_k) &\leq \left( \frac{(n-1)(2p-1)}{2p-n} \|\mathrm{Ric}_-^k\|_p\right)^\frac{1}{2} \beta^{\frac{2p-1}{2p}}+ \tilde{f}(\beta,n,p,k,d,\|\mathrm{Ric}_-^k\|_p),
\end{align*}
\end{theorem}

\begin{remark}
    In particular, $\tilde{f} = 0$ when $\|\mathrm{Ric}_-^k\|_{p} = 0$, which recovers $h_2$ comparison theorem pointwise Ricci curvature lower bound case in \cite{bayle2}.
\end{remark}
We only mention the differences from the proof of Theorem \ref{them:1.1negative} and Theorem \ref{them:1.1}.
\begin{proof}[Proof of Theorem \ref{thm:relative1} and Theorem \ref{thm:relative2}]
    Take any $x \in \partial \Omega$.
    Define a function $I_x(t, g) = |\partial B_x(r) \cap \Om|$ when $t = |B_x(r) \cap \Om|$ and $I_{\bar x} (t, g) = |\partial \bar B (\bar r) \cap \mathbb{H}_k^n|$ when $|\bar B(\bar r)\cap \mathbb{H}_k^n| = t$.
    Let 
    $$U_t =\{\theta \in \SS^{n-1} \, | \, (t,\theta) \in B_x(r)\cap\Om\},$$
    where $(t,\theta)$ is a geodesic polar coordinates centered at $x$.
    Convexity of $\Omega$ implies that if $a \leq b$, then  $U_b \subseteq U_a$.
    Thus,
    \begin{align*}
        \frac{d}{dr} |\partial B_x(r) \cap \Om| & = \lim_{h \to 0} \frac{| \partial B_x (r+h) \cap \Om | - |\partial B_x(r) \cap \Om|}{h}\\
        & = \lim\limits_{h \to 0}\frac{\int_{U_{r+h}} A(r+h, \theta) d \theta - \int_{U_r} A(r,\theta) d \theta}{h}\\
        & \leq \lim_{h \to 0} \frac{\int_{U_r} A(r+h, \theta) d\theta - \int_{U_r} A(r,\theta) d\theta}{h}\\
        & = \int_{U_r} A'(r,\theta) d \theta = \int_{U_r} m(r,\theta) A(r, \theta) d \theta.
    \end{align*}
    Also, it follows from $| B_x(r) \cap \Om| = t$ that
    \begin{align*}
        \frac{dt}{dr} 
        = \frac{d}{dr} \int_0^r \int_{U_s} A(s,\theta) d\theta ds
         = \int_{U_r} A(r,\theta ) d \theta
         = | \partial B_x(r) \cap \Om|.
    \end{align*}
    Those two inequalities imply that
    $$I_x' (t, g) = \frac{d I_x}{dr} \frac{dr}{dt} \leq \frac{\int_{U_r} m(r,\theta) A(r,\theta) d \theta}{ |\partial B_x (r) \cap \Om|}.$$
    Take any $\Bar{x}$ on the boundary of $\mathbb{H}_k^n$.
    Then we have
    $$I_{\bar{x}}'(t, g_k) = \frac{\int_{\SS_{+}^{n-1}} \bar{m}(\bar r) \bar{A}(\bar r) d\theta}{ \frac{1}{2}| \partial B_{\bar x}(\bar r)|} = \bar m (\bar r),$$
    where $\bar r$ is determined by $|B_{\bar x} (\bar r) \cap \mathbb{H}_k^n| = t$.
    Then define, as before $D(t) = I_x (t, g) - I_{\bar{x}} (t, g_k)$, we compute:
\begin{align*}
    D'(t) 
    & = \frac{\int_{U_r} [m(r,\theta)-\bar{m}(r)] A(r,\theta ) d \theta}{| \partial B_x(r) \cap \Om|} +  [\bar{m}(r)-\bar{m}(\bar{r})],
\end{align*}
where $t = |B_x(r)| = |B_{\bar{x}}(\bar{r})|$.
Let $m_+^k(r):= (m(r,\theta) -\bar{m}(r))_+$.
Then
\begin{align}\label{ineq:D}
    D(\beta) & = \int_0^\beta D'(t) d t \nonumber\\
    &\leq \int_0^\beta \frac{\int_{U_r} m_+^k(r) A(r,\theta ) d \theta}{|\partial B_x(r) \cap \Omega|} dt + \int_0^\beta (\bar{m}(r) - \bar{m}(\bar{r})) d t \nonumber\\
    &\leq \int_0^{r_0} \int_{U_r} m_+^k(r) A(r,\theta ) d \theta dr + \int_0^{\beta} (\bar{m}(r) - \bar{m}(\bar{r})) d t,
\end{align}
where $r_0$ is determined by $\beta = |B_x(r_0)\cap \Omega|$.
Using Holder's inequality, we have
\begin{align*}
    \int_0^{r_0}\int_{U_r} m_+^k A(r,\theta ) d \theta dr
    &\leq \left( \int_0^{r_0}\int_{U_r} (m_+^k)^{2p} A(r,\theta ) d \theta dr\right)^\frac{1}{2p} \left( \int_0^{r_0}\int_{U_r}  A(r,\theta ) d \theta dr\right)^\frac{2p-1}{2p}\\
    & =\left( \int_{B_x(r_0)\cap \Om} (m_+^k)^{2p} d vol\right)^\frac{1}{2p}|B_x(r_0)\cap \Om|^\frac{2p-1}{2p}\\
    & \leq \|m_+^k\|_{2p}(r_0) \beta^{\frac{2p-1}{2p}}.
\end{align*}
For the second term in the right hand side of \eqref{ineq:D}, we need a volume comparison on a convex set with integral Ricci curvature:
$$|B_x(r)\cap \Omega| \leq (1+C_1(n,p,k,d) (\|\mathrm{Ric}_-^k\|_p )^\frac{1}{2})^{2p} |B_{\bar{x}} (r)\cap \mathbb{H}_k^n|,$$
where $C_1(n,p,k,d)$ is given by constant multiple of $C(n,p,k,d)$ defined in the proof of Theorem \ref{them:1.1negative} and \ref{them:1.1}.
The above inequality can be obtained by modifying the proof in Chen and Wei \cite{chenwei}.
Now following the proof of Theorem \ref{them:1.1negative} and Theorem \ref{them:1.1}, we obtain:
$$h_2^\Om(\beta,g) - h_2^{\mathbb{H}_k^n}(\beta,\bar{g})  \leq \left( \frac{(n-1)(2p-1)}{2p-n} \|\mathrm{Ric}_-^k\|_p\right)^\frac{1}{2} \beta^{\frac{2p-1}{2p}} + \tilde{f}(\beta, n , p , k , d, \|\mathrm{Ric}_-^k\|_p).$$
\end{proof}

\section{Comparisons for $h_1$ isoperimetric profile}
We are now ready to discuss the differential comparison result for the $h_1$ isoperimetric profile function, as stated in Theorem \ref{thm:supsolution}. Our error term in this context arises from the choice of $\alpha > 1$, as we utilize the Heintze-Kartcher generalization for integral curvature proposed by Petersen and Sprouse (Theorem 4.1, \cite{Petersen-Sprouse}) to replace the conventional Heintze-Kartcher estimate in the argument presented by Ni and Wang in Theorem 2.2 of \cite{Ni}. We recover the pointwise result by selecting $\alpha = 0$ in the limit.
\begin{proof}[Proof of Theorem \ref{thm:supsolution}]
    We initiate our proof by utilizing Aubrey's generalized Bonnet-Myers theorem (Theorem 1.2, \cite{Aubry07}), which establishes that the diameter of manifolds with small integral Ricci curvature is uniformly bounded. This ensures not only the compactness of $M$ but also enables us to later apply Petersen and Sprouse's Heintze-Karcher inequality (Theorem 4.1, \cite{Petersen-Sprouse}). 
    
    Choose any $\beta_0$ and let $U$ be a small neighborhood  of $\beta_0$.
  Let $\beta_0$ be any chosen value, and consider a small neighborhood $U$ around $\beta_0$. Within $U$, let $0 < \psi(\beta) \leq h_1(\beta, g)$ be a smooth function satisfying $\psi(\beta_0) = h_1(\beta_0, g)$. Suppose $\Omega$ is an isoperimetric region with volume $\beta_0$, meaning $\frac{|\partial \Omega|}{|M|} = h_1(\beta_0, g) = \psi(\beta_0)$. Then, $\Omega$ has constant mean curvature $m$ on its smooth part of the boundary $\partial \Omega$. We write:  
    $$m = \psi' (\beta_0),$$
    as obtained from the first variation formula. For a detailed explanation, refer to the proof of Theorem 2.2 in \cite{Ni}.
    
    Define
    \begin{align*}
        r_0 &= \max\{dist(x, \partial \Omega )\, | \, x \in \Omega\}\\
        r_1 &= \max \{dist(x, \partial \Omega) \, | \, x \in M \setminus \Omega\} \leq d -r_0.
    \end{align*}
Given our $\alpha>1$ and $p>\frac{n}{2}$ given we apply now Heintze-Kartcher inequality with integral Ricci curvature bounds (Theorem 4.1, \cite{Petersen-Sprouse}), so we can find $\delta(n,p,\alpha,k)>0$ such that if $\|\mathrm{Ric}_-^k\|_p \leq \delta$,
        then 
    $$\textup{vol} ( \Omega) \leq \alpha \frac{\textup{area}(\partial \Omega)}{\textup{area}(\partial \bar \Omega)} \textup{vol} ( \bar \Omega)$$
and
$$\textup{vol} ( M \setminus \Omega) \leq \alpha \frac{\textup{area}(\partial \Omega)}{\textup{area}(\partial \bar \Omega)} \textup{vol} ( \bar B \setminus \bar \Omega),$$
where $\bar \Omega$ is the ball of radius $r_0$ having the constant mean curvature $m$ and $\bar B $ is the ball of radius $\textup{diam}(M)$ on the model manifold $\SS^n_k$.
By adding two inequalities, we have
$$\textup{vol}(M) \leq  \alpha \frac{\textup{area}(\partial \Omega)}{\textup{area}(\partial \bar \Omega)}\textup{vol}(\bar B).$$
Thus,
$$1 \leq \alpha \cdot \frac{\textup{area}(\partial \Omega)}{\textup{vol}(M)} \cdot \frac{\textup{vol}(\bar B)}{\textup{area}(\partial \bar \Omega)}$$
or equivalently,
\begin{align}\label{ineq:3-1-1}
    1 \leq \alpha \psi(\beta_0)\cdot \frac{\textup{vol}(\bar B)}{\textup{area}(\partial \bar \Omega)}.
\end{align}

We now aim to compute $|\partial \bar \Omega|$. Let's consider the volume element $\bar A(r) = sn_k^{n-1}(r)$ on the model manifold $\mathbb{S}^n_k$. Then, $\bar m(r) = (n-1)\sqrt{k} \cot(\sqrt{k}r)$, where $\bar m(r)$ denotes the mean curvature on the boundary of a ball with radius $r$ on $\SS^n_k$. We have:
$$sn_k(r) = \frac{n-1}{\sqrt{(\bar m(r))^2 + (n-1)^2 k}}.$$
Since $\bar \Omega$ has radius $r_0$ and constant mean curvature $m = \bar m (r_0) = \psi'(\beta_0)$, utilizing the above expression, we obtain:
\begin{align}\label{ineq:3-1-2}
    |\partial \bar \Om| = |\SS^{n-1}| sn_k^{n-1}(r_0) = \left(k + \left(\frac{\psi'(\beta_0)}{n-1}\right)^2\right)^{-\frac{n-1}{2}}.
\end{align}
Note that
\begin{align}\label{ineq:3-1-3}
    |\bar B| = |\SS^{n-1}| \int_0^{d'} sn_k^{n-1}t\, dt,
\end{align}
where $d' = \min \{\pi, \mathrm{diam}(M)\}$.
Thus, combining \eqref{ineq:3-1-1}, \eqref{ineq:3-1-2}, and \eqref{ineq:3-1-3}, we have
\begin{align*}
    1 &\leq \alpha \psi(\beta_0)\cdot \frac{|\bar B|}{|\partial \bar \Omega|}
     = \alpha \psi (\beta_0)\cdot \frac{\int_0^{d'} sn_k^{n-1} t \, dt}{ sn_k^{n-1}(r_0)}\\
    & = \alpha \psi(\beta_0) \left(k + \left(\frac{\psi'(\beta_0)}{n-1}\right)^2\right)^\frac{n-1}{2} \int_0^{d'} sn_k^{n-1}t \, dt.
\end{align*}
Hence we obtain
$$\alpha \psi (\beta_0)\left(k + \left(\frac{\psi'(\beta_0)}{n-1}\right)^2\right)^\frac{n-1}{2} \geq \frac{1}{\int_0^{d'} sn_k^{n-1}t \, dt}.$$
\end{proof}

\begin{proof}[Proof of Theorem \ref{thm:1.3}]
    We first observe that there exists $\delta_1 >0$ such that if $\|\mathrm{Ric}_-^k\|_p<\delta_1$, then $M$ is compact, as ensured by the generalized Bonnet-Myers Theorem with integral Ricci curvature bounds (Theorem 1.2, \cite{Aubry07}). With this established, we proceed to consider a normalized isoperimetric profile function $h_1(\beta,g)$ on a manifold $(M,g)$ with $\|\mathrm{Ric}_-^k\|_p<\delta = \min\{\delta_1,\delta_2\}$, where $\delta_2$ is as defined in Theorem \ref{thm:supsolution}.
    
    In order to get a contradiction, we assume that there is $\beta\in (0,1)$ satisfying
    \begin{equation}\label{ineq:h1-assumption}
        h_1 (\beta, g) < L \phi(\beta) - \epsilon,
    \end{equation}
    where $\phi(\beta) = h_1(\beta, g_k)$.
    Define a function $f$ on $(0,1)$ as
    $$f(\beta) : = \frac{h_1(\beta, g) + \epsilon}{L\phi(\beta)}.$$
    Then by asymptotic of $h_1$ and $\phi$ as in \eqref{eq:asymptotics}, we have 
    \begin{align*}
        \lim\limits_{\beta \to 0^+} f(\beta) & = \lim\limits_{\beta \to 0^+} \frac{h_1(\beta, g) + \epsilon}{L \phi(\beta)} = +\infty.
    \end{align*}
    Since $h_1$ and $\phi$ are symmetric, $f(1- \beta) = f(\beta)$ and so $\lim\limits_{\beta \to 1^-} f(\beta) = +\infty$.
    Then we know that $f$ attain its minimum $\lambda$ at $\beta_0 \in (0,1)$.
    That is,
    $$h_1(\beta_0, g) + \epsilon = \lambda L \phi (\beta_0)$$
    and
    $$h_1(\beta, g) + \epsilon \geq \lambda L \phi (\beta)\quad \forall \beta \in (0,1).$$
    Moreover, our assumption \eqref{ineq:h1-assumption} implies that $\lambda <1$.

    Consider a smooth function $\psi$ such that $0 < \psi (\beta)\leq h_1(\beta,g)$ near $\beta_0$ and $\psi(\beta_0) = h_1(\beta_0, g)$.
    Then we get
    \begin{equation}\label{eq:h1-psi}
        \psi (\beta_0,g) = \lambda L \phi (\beta_0) - \epsilon, \quad \psi'(\beta_0,g) = \lambda L \phi'(\beta_0).
    \end{equation}
    By Theorem \ref{thm:supsolution} with a support function $\psi$, we have
    \begin{align}\label{ineq:h1-support}
        \frac{1}{\lambda_{n,d'}^k} &\leq \alpha \psi \left( k+ \left( \frac{\psi'}{n-1}\right)^2 \right)^{\frac{n-1}{2}}.
    \end{align}
    At $\beta_0$, the inequalities \eqref{eq:h1-psi} and \eqref{ineq:h1-support} imply that
    \begin{align}\label{ineq:temp1}
        \frac{1}{\lambda_{n,d'}^k} &\leq \alpha (\lambda L \phi - \epsilon) \left( k+ \left( \frac{\lambda L \phi'}{n-1}\right)^2 \right)^{\frac{n-1}{2}}\nonumber\\
        & <\alpha ( L \phi - \epsilon) \left( k+ \left( \frac{ L \phi'}{n-1}\right)^2 \right)^{\frac{n-1}{2}},
    \end{align}
    where the last inequality comes from $L \geq 1$, $\phi(\beta_0)>0$, $\lambda <1$, and $\alpha >1$.
    Let us note that $\phi$ is a solution to the differential equation
    \begin{equation}\label{eq:h1-phi}
        \phi \left( k+ \left( \frac{\phi'}{n-1}\right)^2 \right)^{\frac{n-1}{2}} = \frac{1}{\gamma_n}.
    \end{equation}
    Then combining inequalities \eqref{ineq:temp1} and \eqref{eq:h1-phi}, we get
    \begin{align*}
        \alpha ( L \phi - \epsilon) \left( k+ \left( \frac{ L \phi'}{n-1}\right)^2 \right)^{\frac{n-1}{2}} &>\frac{\gamma_n}{\lambda_{n,d'}^k} \phi \left( k+ \left( \frac{\phi'}{n-1}\right)^2 \right)^{\frac{n-1}{2}}\\
        & = L^n \phi \left( k+ \left( \frac{\phi'}{n-1}\right)^2 \right)^{\frac{n-1}{2}}\\
        & = L \phi \left( L^2k+ \left( \frac{L\phi'}{n-1}\right)^2 \right)^{\frac{n-1}{2}}\\
        &\geq L \phi \left( k+ \left( \frac{L\phi'}{n-1}\right)^2 \right)^{\frac{n-1}{2}}.
    \end{align*}
    Thus,
    $\alpha ( L\phi - \epsilon) > L \phi$
    and
    $$\epsilon < \frac{(\alpha -1) L \phi}{\alpha} \leq \frac{\alpha-1}{\alpha} \frac{L}{\int_0^\pi \sin(\sqrt{k}t) \,dt},$$
    which is a contradiction.
\end{proof}